\documentclass[10pt]{article} 
\usepackage[preprint]{tmlr}

\usepackage{amsmath,amsfonts,bm}

\def\eqref#1{equation~\ref{#1}}

\def\1{\bm{1}}

\DeclareMathAlphabet{\mathsfit}{\encodingdefault}{\sfdefault}{m}{sl}
\SetMathAlphabet{\mathsfit}{bold}{\encodingdefault}{\sfdefault}{bx}{n}

\DeclareMathOperator*{\argmax}{arg\,max}

\usepackage{hyperref}
\usepackage{url}

\usepackage{amsthm}
\usepackage{booktabs} 
\usepackage{tikz} 
\usepackage{subcaption}
\usepackage{enumerate}
\usepackage{comment}

\usepackage[mathscr]{euscript}
\usepackage{algorithm}
\usepackage{algorithmic}

\usepackage{amsfonts}
\usepackage{amssymb}
\usepackage{amsthm}
\usepackage[parfill]{parskip}
\usepackage{tcolorbox}
\tcbuselibrary{theorems}
\usepackage{amsmath}
\usepackage{mdframed}

\usepackage{pgfplots}
\pgfplotsset{width=7.5cm,compat=1.18}

\usepackage{enumerate}
\usepackage{mathrsfs}
\usepackage{cleveref}
\usepackage{url}

\usepackage{cancel}

\usetikzlibrary{positioning}

\usepackage{diagbox}

\usepackage{stmaryrd}
\usepackage{cleveref}
\usepackage{xcolor}
\usepackage{multirow}
\usepackage{pifont}
\theoremstyle{plain}
\newtheorem{theorem}{Theorem}[section]
\newtheorem{lemma}[theorem]{Lemma}

\newtheorem{proposition}[theorem]{Proposition}
\newtheorem{property}[theorem]{Property}

\theoremstyle{definition}
\newtheorem{definition}{Definition}[section]
\newtheorem{example}[definition]{Example}

\newtheorem{remark}[definition]{Remark}

\newmdtheoremenv[backgroundcolor=yellow!35, hidealllines=true]{shadedlemma}{Lemma}

\newmdtheoremenv[backgroundcolor=green!20, hidealllines=true]{shadeddefinition}{Definition}

\newmdtheoremenv[backgroundcolor=green!20, hidealllines=true]{shadedassumption}{Assumption}

\title{Beyond Log-Concavity: Theory and Algorithm for\\ Sum-Log-Concave Optimization}

\author{\name Mastane Achab \email mastane.achab@tii.ae \\
      \addr Technology Innovation Institute, 9639 Masdar City, Abu Dhabi, United Arab Emirates
      }

\def\month{MM}  
\def\year{YYYY} 
\def\openreview{\url{https://openreview.net/forum?id=XXXX}} 

\begin{document}

\maketitle

\begin{abstract}

This paper extends the classic theory of convex optimization to the minimization of functions that are equal to the negated logarithm of what we term as a ``sum-log-concave'' function, i.e., a sum of log-concave functions.
In particular, we show that such functions are in general not convex but still satisfy generalized convexity inequalities. 
These inequalities unveil the key importance of a certain vector that we call the ``cross-gradient'' and that is, in general, distinct from the usual gradient.
Thus, we propose the Cross Gradient Descent (XGD) algorithm moving in the opposite direction of the cross-gradient and derive a convergence analysis.
As an application of our sum-log-concave framework, we introduce the so-called ``checkered regression'' method relying on a sum-log-concave function. 
This classifier extends (multiclass) logistic regression to non-linearly separable problems since it is capable of tessellating the feature space by using any given number of hyperplanes, creating a checkerboard-like pattern of decision regions.

\end{abstract}

\section{Introduction}\label{sec:intro}

In the landscape of machine learning and optimization, the quest to understand the mechanics of gradient descent algorithms, especially in the non-convex scenario, has received much attention. Traditional global convergence analysis of these algorithms has been confined mostly to convex settings (\cite{boyd2004convex},\cite{bach2021learning}). Despite recent advances, our understanding of the performance of these algorithms beyond the convex setting is rather limited, with most results being confined to local convergence analyses \citep{danilova2022recent}.

Deep learning methods, such as deep neural networks, have demonstrated remarkable empirical successes in tackling non-linear and high-dimensional problems, ranging from image or speech recognition to natural language processing (\cite{lecun2015deep}, \cite{goodfellow2016deep}). These techniques excel in creating complex decision boundaries and robust feature representations, efficiently capturing intricate patterns in the data. However, while they are empirically proficient, the theoretical underpinnings, particularly a rigorous global convergence analysis of the associated optimization methods, remain an active area of investigation. Most existing theories focus on specific architectures or rely on idealized assumptions that might not hold in practice (\cite{jacot2018neural},\cite{chizat2018global}). Therefore, while deep learning methods continue to dominate the field due to their powerful performance, the lack of comprehensive theoretical guarantees for their convergence behaviors warrants further research.

In this paper, we extend the conventional global convergence analysis of gradient descent, taking it way beyond the classical convex case. We examine functions represented as the negated logarithm of what we call a \emph{sum-log-concave} function, that is a sum of log-concave functions. Our analysis reverts to the traditional convex case when applied to a single log-concave function. In the more general scenario, we are still able to provide an upper bound for the excess risk of order $\tilde{\mathcal{O}}(1/\sqrt{T})$, where $T$ represents the number of steps of our proposed algorithm.

Furthermore, we introduce a new sum-log-concave model that generalizes (multiclass) logistic regression, suitable for non-linearly separable problems. For binary classification, our model is capable of tessellating the feature space with any given number of hyperplanes. The decision regions, outputting label 0 or 1, follow a checkerboard-like pattern, generalizing the half-spaces and the linear prediction rule learned by logistic regression (see e.g. \cite{bishop2006pattern}, \cite{hastie2009elements}).
For that reason, we call this new method the \emph{checkered regression model} and we define it in a computationally efficient
way through the circular convolution of several SoftArgMax vectors.

In the binary classification case, our checkered regression model can be seen as a single hidden-layer neural network with $\tanh$ activations that are multiplied together, instead of being additively combined as is customary in classic deep learning architectures.
In particular, when multiplying only two hyperbolic tangents, our model can be interpreted as a smooth XOR with continuous sigmoids in place of Boolean variables.


\paragraph{Key contributions} Our main contributions can be summarized as follows:
\begin{enumerate}
    \item We prove a convergence guarantee of our proposed ``cross gradient descent'' algorithm holding for any function equal to the negated logarithm of a sum-log-concave function.
    \item We introduce a new sum-log-concave model (namely ``the checkered regression'') that naturally extends multiclass logistic regression beyond
    linearly separable problems.
\end{enumerate}

\paragraph{Notations}
The Euclidean norm of any vector $v \in \mathbb{R}^p$ ($p\ge 1$) is denoted $\| v \|$
and, for any $r>0$, let $\mathcal{B}(v, r) = \{ v' \in \mathbb{R}^p : \| v'-v \|\le r \} $ denote the Euclidean ball centered at $v$ with radius $r$.
Given two vectors $u= (u_0,\dots, u_{p-1}) \in \mathbb{R}^p$ and $v= (v_0,\dots, v_{p-1}) \in \mathbb{R}^p$, their circular convolution $ u \circledast v$ is another $p$-dimensional vector with $k$-th entry equal to $\sum_{ 0\le i,j\le p-1: i+j \equiv k [p] } u_i v_j$,
for any $k\in \{0,\dots, p-1\}$.
For any integer $S\ge 1$, we denote by
$\mathfrak{S}_S$ the set of permutations of $\{1,\dots,S\}$ and by $\Delta_S$ the probabilty simplex:
\begin{equation*}
    \Delta_S = \left\{ \mu=(\mu_1,\dots,\mu_S) \in [0,1]^S : \mu_1 +\dots +\mu_S = 1 \right\} \ .
\end{equation*}
The Kullback-Leibler divergence \cite{kullback1951information} will be denoted by ``$D_{\text{KL}}$'' throughout the paper: for any $\mu, \nu \in \Delta_S$,
$ D_{\text{KL}}( \mu \| \nu )
= \sum_{s=1}^S \mu_s \log\left( \frac{\mu_s}{\nu_s} \right) $.
Given $K\ge 2$ vectors $v^{(1)}=(v^{(1)}_1, \dots, v^{(1)}_{d_1}), \dots, v^{(K)}=(v^{(K)}_1, \dots, v^{(K)}_{d_K})$, their outer product, denoted $v^{(1)}\otimes \dots \otimes v^{(K)}$, is the $K$-way tensor defined such that $[v^{(1)}\otimes \dots \otimes v^{(K)}]_{j_1,\dots,j_K} = v^{(1)}_{j_1} \times \dots \times v^{(K)}_{j_K}$, for all $(j_1,\dots,j_K) \in \{1,\dots,d_1\}\times \dots \times \{1,\dots,d_K\}$. 
For $\upsilon=(\upsilon_1,\dots,\upsilon_m) \in \{0,\dots,c-1\}^m$ (with $c\ge 2$), we denote $|\upsilon|=\sum_{k=1}^m \upsilon_k$.

\section{Motivation}
\label{sec:motiv}

\subsection{Ubiquitous log-concave functions in machine learning}

In preamble, let us recall that a function $p : \mathbb{R}^m \rightarrow (0, \infty)$ is \emph{log-concave} if $\log(p)$ is concave (or equivalently, if $-\log(p)$ is convex), i.e. if for all $x,y \in \mathbb{R}^m$, for all $\lambda \in [0, 1]$,
\begin{equation}
    \label{eq:logconcave}
    p( \lambda x + (1-\lambda)y )
    \ge p(x)^\lambda \cdot p(y)^{1-\lambda} \ .
\end{equation}
Typical examples of log-concave functions that are used in machine learning include:
\begin{itemize}
    \item The sigmoid function defined for all $x\in \mathbb{R}$ as $\sigma(x)=(1+e^{-x})^{-1} \in (0,1)$, that is used in logistic regression.
    \item Given $c\ge 2$, each component $\sigma_j(z)$ of the vector-valued ``SoftArgMax'' function 
    \begin{equation}
    \label{eq:softargmax}
    \vec \sigma : z=(z_0,\dots,z_{c-1}) \in \mathbb{R}^c \mapsto \begin{pmatrix}
        \sigma_0(z) \\
        \vdots \\
        \sigma_{c-1}(z)
    \end{pmatrix} 
    =
    \begin{pmatrix}
        \frac{e^{z_0}}{ \sum_{k=0}^{c-1} e^{z_k} } \\
        \vdots \\
        \frac{e^{z_{c-1}}}{ \sum_{k=0}^{c-1} e^{z_k} }
    \end{pmatrix} \in \Delta_c \ ,
    \end{equation}
    that is used to produce a categorical probability distribution in multiclass logistic regression.

    \item The Gaussian bell curve function
    $f(x) = e^{-x^2/2}$ ($\forall x\in \mathbb{R}$)
    is log-concave since the quadratic loss function $-\log(f(x)) = x^2/2$ is convex.
\end{itemize}

For all three types of log-concave functions described above, taking the negative logarithm (a.k.a. ``log-loss'') yields the loss function of a popular machine learning method, respectively: logistic regression, multi-class logistic regression, and least squares linear regression.

Next, we provide motivating examples for the sum-log-concave assumption studied in this paper.
One shall notice that this assumption is frequently encountered in the machine learning literature through the Gaussian mixture model (GMM).
Indeed, the probability density of a GMM is a sum of log-concave Bell curve functions.

\subsection{Motivating examples of sum-log-concave methods}

\begin{example}[SoftMin Regression]
\label{ex:softmin}
Given an input vector $x\in \mathbb{R}^d$, and for each $1\le s \le S$, $K_s\ge 1$ targets $y_{s,1},\dots,y_{s,K_s} \in \mathbb{R}$, 
and parameters $\omega_{s,1},\dots, \omega_{s,K_s} \in \mathbb{R}^d$ (all collected in ``$\omega$''), 
we define the ``SoftMin regression'' loss function as follows:
    \begin{equation*}
        F(\omega) = -\log\left( \sum_{s=1}^S e^{- \frac12 \sum_{k=1}^{K_s} ( y_{s,k} - x^\top w_{s,k} )^2} \right) \ ,
    \end{equation*}
    which is equal to the log-loss of a sum of log-concave functions.
If $S=1$ and $K_1=1$, this corresponds to a least squares linear regression.
In the general case, this objective smoothly mimicks the minimum of the square losses:
\begin{equation*}
    F(\omega)
    \lessapprox \min_{1\le s\le S} \frac12 \sum_{k=1}^{K_s} ( y_{s,k} - x^\top w_{s,k} )^2  \ .
\end{equation*}

\end{example}

\begin{example}[Smooth XOR]
\label{ex:smoothXOR}
Given two Boolean variables $A,B \in \{0,1\}$, we recall that the XOR logical gate is given by: $\text{XOR}(A,B)=A(1-B) + (1-A)B$.
By replacing these Boolean variables by continuous sigmoids, namely $A=\sigma(a), B=\sigma(b)$ with $a,b\in \mathbb{R}$ and $\sigma(x)=(1+e^{-x})^{-1}=1-\sigma(-x)$, we obtain a sum-log-concave smooth XOR function:
\begin{equation*}
    \chi(a,b) = \sigma(a)\sigma(-b) + \sigma(-a)\sigma(b)
    = \frac{e^{-a} + e^{-b}}{1+e^{-a}+e^{-b}+e^{-a-b}} \ .
\end{equation*}
In fact, this function naturally arises as the posterior probability of binary labelled data sampled from mixtures of distributions belonging to some exponential family. For instance in the Gaussian case, consider a random pair $(X,Y)$ valued in $\mathbb{R}^d \times \{-1, +1\}$ such that $\mathbb{P}(Y=1)=\frac12$ and
\begin{equation*}
    \begin{cases}
        [ X | Y=+1 ] \sim \frac12 \mathcal{N}(\mu_0, I_d) + \frac12 \mathcal{N}(\mu_0+\mu_1+\mu_2, I_d) \\
        [ X | Y=-1 ] \sim \frac12 \mathcal{N}(\mu_0+\mu_1, I_d) + \frac12 \mathcal{N}(\mu_0+\mu_2, I_d) \ ,
    \end{cases}
\end{equation*}
such that $\langle \mu_1, \mu_2 \rangle = 0$ (which is known as the ``XOR Gaussian mixture model'', see e.g. \cite{ben2022high}).
Then, a direct application of Bayes' rule shows that the posterior class probability writes as a smooth XOR with affine arguments with respect to any data point $x$:
\begin{equation*}
    \mathbb{P}(Y=1 | X=x) = \chi( w_1^\top x + b_1 , w_2^\top x + b_2) \ ,
\end{equation*}
where $w_1,w_2 \in \mathbb{R}^d$ and $b_1,b_2 \in \mathbb{R}$ are constants depending on the Gaussian mean parameters $\mu_0, \mu_1,\mu_2$.
This smooth XOR model will be further generalized in section \ref{sec:CR}.
\end{example}

As seen in Example \ref{ex:smoothXOR}, the smooth XOR function $\chi$ is a sum of two log-concave functions, which makes it fundamentally distinct from the well-studied class of log-concave functions. For instance, any log-concave function $p$ (see Eq. \ref{eq:logconcave}) is also quasi-concave (i.e. for all $x,y$, for all $\lambda \in [0,1]$, $p(\lambda x + (1-\lambda y)) \ge \min( p(x), p(y))$).
Nevertheless, the smooth XOR function is not quasi-concave: indeed, 
\begin{equation*}
    \chi\left( \frac{1 + 0}{2}, \frac{0+1}{2} \right)
    \approx 0.47 < \chi(1,0)=\chi(0,1) = 0.5 \ .
\end{equation*}

\subsection{Learning rock-paper-scissors}


We now illustrate on a simple example that sum-log-concave models allows to learn richer patterns than standard log-concave methods.
More specifically, we consider the smooth XOR classifier that we define through the smooth XOR function $\chi$ introduced in Example \ref{ex:smoothXOR} together with parameters vectors $\omega_1,\omega_2 \in \mathbb{R}^d$ (with no bias parameters). Given an input vector $x \in \mathbb{R}^d$, our proposed classifier outputs the probability: $\chi(\omega_1^\top x, \omega_2^\top x) \in (0,1)$.
We seek to minimize (with respect to $\omega_1,\omega_2$) the corresponding log-loss over the dataset consisting of the three possible duels in the rock-paper-scissors game.
The rules of this game form a cycle as illustrated in Figure \ref{fig:rock-paper-scissors}.
Formally, rock-paper-scissors can be represented as an intransitive distribution $P$ over $\mathfrak{S}_3$ with pairwise marginals $p_{ab} = \mathbb{P}_{\Sigma \sim P}( \Sigma(a) < \Sigma(b) )$ (for $a \neq b \in \{ \text{rock},\text{paper},\text{scissors} \}$) that violate the standard stochastic transitivity assumption:
$ p_{ab} \ge \frac{1}{2} \  \text{and} \ p_{bc} \ge \frac{1}{2} \  \xcancel{\Longrightarrow} \  p_{ac} \ge \frac{1}{2} $, as shown in Figure \ref{tab:rps_probabilities}.
Our smooth XOR method is compared against the classic Bradley-Terry model (\cite{bradley1952rank}) ; both methods are learned via gradient descent\footnote{The code can be found here: https://gist.github.com/mastane/rock-paper-scissors.ipynb}. 
Figure \ref{fig:rock-paper-scissors_learning} shows that the smooth XOR approach successfully learns the cyclic pairwise probabilities with a log-loss tending towards zero, while the Bradley-Terry method cannot do better than predicting equal probability $\frac12$ for each duel and has a log-loss that converges to $\log(2)$.

\begin{figure}[h]
\centering
\begin{minipage}[t]{0.45\textwidth}
\begin{tikzpicture}[->,>=to,shorten >=1pt,auto,node distance=2cm,
                    thick,main node/.style={circle,fill=blue!20,draw,font=\sffamily\Large\bfseries}, scale=0.8, every node/.style={transform shape}]

  \node[main node] (1) {Rock};
  \node[main node] (2) [below left=of 1] {Scissors};
  \node[main node] (3) [below right=of 1] {Paper};

  \path[every node/.style={font=\sffamily\small}, line width=.9mm]
    (1) edge [bend right] node[left] {} (2)
    (2) edge [bend right] node[below] {} (3)
    (3) edge [bend right] node[right] {} (1);

\end{tikzpicture}

\caption{Cyclic representation of the rules of rock-paper-scissors. In every possible duel, the arrow starts from the winner and points towards the loser.}
\label{fig:rock-paper-scissors}
\end{minipage}%
\hfill
\begin{minipage}[t]{0.45\textwidth}
\centering
\begin{tabular}{|c|c|c|c|}
\hline
 & Rock & Paper & Scissors \\ \hline
Rock & \diagbox{}{} & $p_{rp} = 0$ & $p_{rs} = 1$ \\ \hline
Paper & $p_{pr} = 1$ & \diagbox{}{} & $p_{ps} = 0$ \\ \hline
Scissors & $p_{sr} = 0$ & $p_{sp} = 1$ & \diagbox{}{} \\ \hline
\end{tabular}
\caption{Pairwise marginal probabilities for rock-paper-scissors game.}
\label{tab:rps_probabilities}
\end{minipage}

\end{figure}

\begin{figure}[t]
  \centering
  \includegraphics[width=0.7\textwidth]{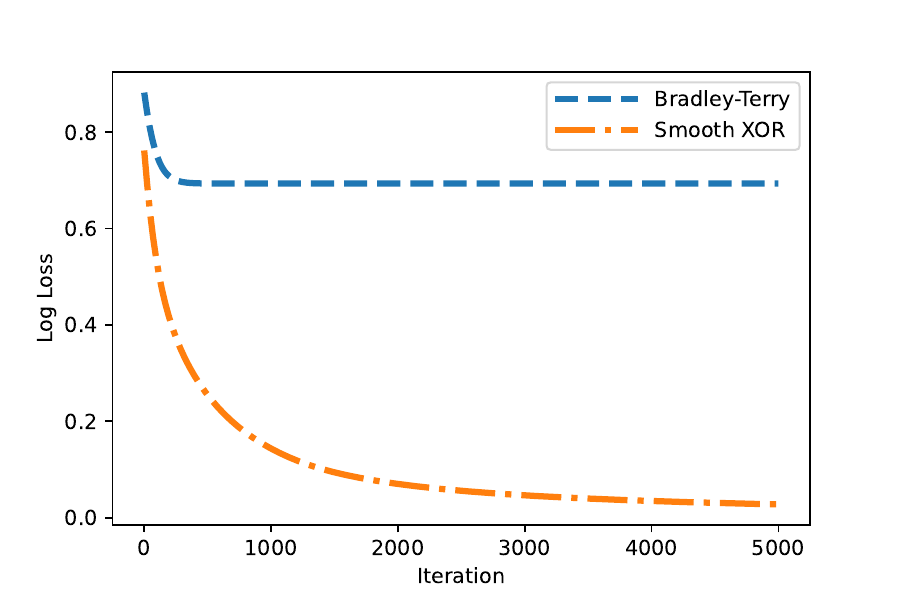}
  \caption{Learning dynamics of the Bradley-Terry model (only allowing stochastically transitive ranking distributions) versus our proposed sum-log-concave smooth XOR method at learning rock-paper-scissors rules.
  Each of the three items (rock, paper, scissors) is represented by a one-hot encoders of size 3.
  The dataset is composed of $\binom{3}{2} = 3$ pairwise comparisons corresponding to all possible duels.
  Both models are learned via gradient descent over 5000 epochs with learning rate set to 0.01 for the log-loss.
  The curves are averaged over $100$ seeds with every parameter independently initialized from the standard normal distribution. 
  }
  \label{fig:rock-paper-scissors_learning}
\end{figure}

Although our proposed smooth XOR classifier allows to learn cyclic rankings as shown in Figure \ref{fig:rock-paper-scissors_learning}, we stress the importance of the utilized random initialization of the weights. Indeed, since the log-loss of this classifier has a saddle point at zero (i.e., $- \nabla \log(\chi) (0,0) = 0$), it cannot be learned successfully by gradient descent if we initialize its weights with all-zeros vectors. On the other hand, since the Bradley-Terry model has a convex log-loss, it can be learned by gradient descent independently of the initial weights (even though its representational power remains limited to acyclic rankings).
This observation raises an algorithmic challenge, namely that of designing appropriate optimization algorithms for sum-log-concave methods that can escape from saddle-points and that can be learned for any initialization.
Before tackling this challenge in section \ref{sec:xgd}, we first introduce in the next section the class of functions that we propose to minimize and discuss some of their properties.

\section{Main assumption}
\label{sec:ass}

In this section, we introduce our proposed generalization of the notion of log-concavity, termed ``sum-log-concavity'', that is obtained by summing several log-concave functions.

\subsection{Sum-log-concavity} 
We now introduce the main assumption about the class of functions that we seek to minimize in the present study,
namely any function that is equal to the log-loss of a sum-log-concave model.

\begin{shadedassumption}[Negated Log of Sum-Log-Concave]
  \label{ass:sumLogConcave}
  There exists $S\ge 1$ such that the function $F : \mathbb{R}^m \mapsto \mathbb{R}$ is given by: 
  \begin{equation*}
  \forall \theta \in \mathbb{R}^m \ , \quad
    F(\theta) = - \log\left( \sum_{s=1}^S p_s(\theta) \right) \ ,
  \end{equation*}
  where each $p_s : \mathbb{R}^m \mapsto (0, \infty)$ is a differentiable log-concave function.
  Given $m\ge 1$, we denote by $\mathfrak{X}_m$ the family of such functions $F$.
\end{shadedassumption}

We point out that most of the notions and results presented in this paper for functions satisfying Assumption \ref{ass:sumLogConcave}
naturally extend to functions belonging to the conical hull of $\mathfrak{X}_m$.
For notational convenience and clarity of exposition, we restrict our attention to $\mathfrak{X}_m$.

\subsection{Closure properties}

As for the class of convex functions, the set $\mathfrak{X}_m$ (resp. $\mathfrak{X} := \bigcup_{m = 1}^{\infty} \mathfrak{X}_m$) is closed under summation (resp. affine reparameterization).
These closure properties are highly desirable in machine learning, where one typically aims at minimizing an objective function that writes as a sum of functions: $\mathcal{L}(A,b)=\sum_{i=1}^n f_i(A z_i + b)$ with $n$ the size of the dataset $(z_1,\dots,z_n)$, and where the argument of each $f_i$ is an affine transformation of the input $z_i$.

\begin{proposition}[Summation/Affine Closure]
\label{prop:closure}
The Assumption \ref{ass:sumLogConcave} is closed under summation and affine reparameterization.
    \begin{enumerate}[(i)]
        \item Summation. If $F_1,F_2 \in \mathfrak{X}_m$ , then $F_1 + F_2 \in \mathfrak{X}_m$.
        \item Tensor representation. Let $K\ge 1$ and, for each $k\in\{1,\dots,K\}$, $F_k = -\log(\sum_{s=1}^{S_K} p_{s}^{(k)}) \in \mathfrak{X}_m$ with $S_k\ge 1$. Then,
        \begin{equation*}
            F_1 +\dots + F_K = -\log\left( \sum_{1\le s_1\le S_1,\dots, 1\le s_K\le S_K} \tilde p_{s_1,\dots,s_K} \right) \quad \text{with} \quad \tilde p_{s_1,\dots,s_K} = p_{s_1}^{(1)} \times \dots \times p_{s_K}^{(K)} \ .
        \end{equation*}
        \item Affine reparameterization. If $F \in \mathfrak{X}_m$, $A\in \mathbb{R}^{m\times p}, b\in \mathbb{R}^{m}$, then $\tilde{F} : z\in\mathbb{R}^p \mapsto F(A z + b) $ is in $\mathfrak{X}_p$.
    \end{enumerate}
\end{proposition}

\begin{proof}
\textbf{(i) Summation.}
The summation closedness property of $\mathfrak{X}_m$ is a simple consequence of the fact that the product of two log-concave functions is still log-concave.
Indeed, if $F_1 = -\log(\sum_{s=1}^{S_1} p_s)$ and $F_2 = -\log(\sum_{s=1}^{S_2} q_s)$, then $F_1 + F_2 = -\log(\sum_{s=1}^{S_1} \sum_{s'=1}^{S_2} p_s q_{s'})$ involves a sum of $S_1 \times S_2$ log-concave functions.

\textbf{(ii) Tensor representation.}
We have:
\begin{equation*}
    F_1 +\dots + F_K = -\log\left( \prod_{k=1}^K \left( \sum_{s=1}^{S_k} p^{(k)}_s \right) \right)
    = -\log\left( \sum_{1\le s_1\le S_1,\dots, 1\le s_K\le S_K} p_{s_1}^{(1)} \times \dots \times p_{s_K}^{(K)} \right) \ .
\end{equation*}

\textbf{(iii) Affine reparameterization.}
Follows from the invariance of the notion of concavity under affine reparameterization: indeed, the concavity of $\log(p_s)$ implies that $z\in\mathbb{R}^p \mapsto \log p_s(A z + b)$ is concave.
Hence, $z\in\mathbb{R}^p \mapsto p_s(A z + b)$ is log-concave and thus $\tilde F \in \mathfrak{X}_p$.



\end{proof}

The summation closure property Proposition \ref{prop:closure}-(i) generalizes the fact that a sum of convex functions is still convex. For instance, this is not true for the class of quasi-convex functions that is studied in \cite{hazan2015beyond}.

\section{The cross gradient descent algorithm}
\label{sec:xgd}

In this section, we introduces a novel algorithm tailored to sum-log-concave optimization.
We start by extending the standard notion of convexity to our sum-log-concave framework.
In particular, we will see that a specific vector, that we call ``cross gradient'', plays a role similar to the classic gradient in convex optimization.
This will motivate our new ``Cross Gradient Descent'' algorithm (shortened ``XGD'') for which we provide a convergence analysis.

\subsection{Cross gradients}
\label{subsec:cross_grad}

Let us introduce the notion of cross gradient that will appear (in place of the standard gradient) in our generalized convexity inequalities.

\begin{shadeddefinition}[Cross gradient]
    \label{def:cross_grad}
    Let $F : \mathbb{R}^m \mapsto \mathbb{R}$ be a function satisfying Assumption \ref{ass:sumLogConcave}.
    \begin{enumerate}[(i)]
        \item For any $\theta, \eta \in \mathbb{R}^m$, we define the ``cross-gradient of $F$ at $\theta$ seen from $\eta$'' as follows:
    \begin{equation*}
        \nabla^{\eta} F(\theta) = - \sum_{s=1}^S \frac{p_{s}(\eta)}{ \sum_{s'} p_{s'}(\eta) } \frac{\nabla p_s( \theta ) }{p_s( \theta )} \ .
    \end{equation*}
    \item Similarly for any probability law $\mu=(\mu_1,\dots,\mu_S)\in \Delta_S$, we overload the cross gradient notation by defining:
    \begin{equation*}
        \nabla^{\mu} F(\theta) = - \sum_{s=1}^S \mu_s \frac{\nabla p_s( \theta ) }{p_s( \theta )} \ .
    \end{equation*}
    \end{enumerate}
\end{shadeddefinition}

In particular for $\eta=\theta$, the cross gradient is simply equal to the usual gradient: $\nabla^{\theta} F(\theta) = \nabla F(\theta)$.
In the ``trivial'' case $S=1$, the function $F$ is convex and its cross gradient is always equal to the classic gradient: $\nabla^{\eta} F(\theta) = \nabla F(\theta)$ for all $\theta,\eta \in \mathbb{R}^m$.
In the general case, notice that $\nabla^{\eta} F : \theta \mapsto \nabla^{\eta} F(\theta) $ can be viewed as a vector field parameterized by $\eta$.
We call this new operation ``cross gradient'' because it blends the influence of both points $\theta$ and $\eta$, and, as will be shown in the proof of Lemma \ref{lem:cross_cvx}, it naturally derives from a cross entropy term.
Furthermore, the cross gradient corresponds to an average of the gradients of the ``partial losses'' $\ell_s(\theta) = -\log(p_s(\theta))$:
indeed, $\nabla \ell_s(\theta) = - \frac{\nabla p_s(\theta)}{ p_s(\theta)}$.
The cross gradient operator verifies an important linearity property as explained below. The proof is also provided here as it involves insightful manipulation of functions belonging to the class $\mathfrak{X}_m$.

\begin{proposition}[Linearity]
\label{prop:linearity}
\begin{enumerate}[(i)]
    \item Given some reference point $\eta \in \mathbb{R}^m$, the cross gradient operator $\nabla^\eta$ is linear in the sense that if $F, G \in \mathfrak{X}_m$, then $\nabla^\eta ( F + G ) = \nabla^\eta F + \nabla^\eta G $.
    \item Let $K\ge 1$ and, for each $k\in\{1,\dots,K\}$, $F_k = -\log(\sum_{s=1}^{S_K} p_{s}^{(k)}) \in \mathfrak{X}_m$ with $S_k\ge 1$, and $\mu^{(k)} \in \Delta_{S_k}$. Then,
    \begin{equation*}
        \nabla^{\mu^{(1)} \otimes \dots \otimes \mu^{(K)}}( F_1 + \dots + F_K ) = \nabla^{\mu^{(1)}} F_1 + \dots + \nabla^{\mu^{(K)}} F_K \ ,
    \end{equation*}
    where we use the tensor representation from Proposition \ref{prop:closure}-(ii) for the sum $\sum_{k} F_k$,
    and the outer product $\mu^{(1)} \otimes \dots \otimes \mu^{(K)}$ can be interpreted as an element of the probability simplex $\Delta_{S_1\times \dots \times S_K}$.
\end{enumerate}
\end{proposition}

\begin{proof}
\textbf{(i) Given $\eta$.}
    As $F,G\in \mathfrak{X}_m$, they can be written as follows:
    \begin{equation*}
        \begin{cases}
        F(\theta) = -\log\left( \sum_{s=1}^S p_s(\theta) \right) \\
        G(\theta) = -\log\left( \sum_{r=1}^{R} q_r(\theta) \right)
        \end{cases} ,
    \end{equation*} 
where $p_s, q_r$ are log-concave functions.
Then, the sum also belongs to $\mathfrak{X}_m$ and involves a sum of $S\times R$ log-concave functions:
\begin{equation*}
    F(\theta) + G(\theta) 
    = -\log\left( \left( \sum_{s=1}^S p_s(\theta) \right) \left( \sum_{r=1}^{R} q_{r}(\theta) \right) \right)
    = -\log\left( \sum_{s=1}^S \sum_{r=1}^{R} p_s(\theta) q_{r}(\theta) \right) \ .
\end{equation*}
Finally, we deduce that the cross gradient of $F+G$ is equal to
\begin{multline*}
    \nabla^\eta ( F + G )(\theta)
    = - \sum_{s=1}^S \sum_{r=1}^{R} \frac{p_s(\eta) q_{r}(\eta)}{ \left(\sum_{s'} p_{s'}(\eta) \right) \left( \sum_{r'} q_{r'}(\eta) \right) } \left[ \frac{\nabla p_s( \theta ) }{p_s( \theta )}
    + \frac{\nabla q_r( \theta ) }{q_r( \theta )} \right] \\
    = - \underbrace{\sum_{r=1}^{R} \frac{q_r(\eta) }{ \sum_{r'} q_{r'}(\eta) } }_{1} \sum_{s=1}^S \frac{p_s(\eta) }{ \sum_{s'} p_{s'}(\eta) }  \frac{\nabla p_s( \theta ) }{p_s( \theta )} 
    - \underbrace{\sum_{s=1}^{S} \frac{p_s(\eta) }{ \sum_{s'} p_{s'}(\eta) } }_{1} \sum_{r=1}^R \frac{q_r(\eta) }{ \sum_{r'} q_{r'}(\eta) }  \frac{\nabla q_r( \theta ) }{q_r( \theta )} 
    = \nabla^\eta F(\theta) + \nabla^\eta G(\theta) \ ,
\end{multline*}
where we have used the fact that $\nabla \log(p_s q_r) = \nabla \log(p_s) + \nabla \log(q_r) 
= \frac{\nabla p_s }{p_s} + \frac{\nabla q_r }{q_r} $ .

\textbf{(ii) Given $\mu$'s.}
By the tensor representation from Proposition \ref{prop:closure}-(ii),
\begin{equation*}
        F_1 +\dots + F_K = -\log\left( \sum_{1\le s_1\le S_1,\dots, 1\le s_K\le S_K} \tilde p_{s_1,\dots,s_K} \right) \quad \text{with} \quad \tilde p_{s_1,\dots,s_K} = p_{s_1}^{(1)} \times \dots \times p_{s_K}^{(K)} \ .
\end{equation*}
Hence,
\begin{multline*}
\nabla^{\mu^{(1)} \otimes \dots \otimes \mu^{(K)}}( F_1 + \dots + F_K )(\theta)
= - \sum_{1\le s_1\le S_1,\dots, 1\le s_K\le S_K}
\mu^{(1)}_{s_1}\times \dots \times \mu^{(K)}_{s_K} \left[ \frac{\nabla p^{(1)}_{s_1}(\theta)}{p^{(1)}_{s_1}(\theta)} + \dots + \frac{\nabla p^{(K)}_{s_K}(\theta)}{p^{(K)}_{s_K}(\theta)} \right] \\
= - \sum_{k=1}^K \prod_{j\neq k} \underbrace{ \sum_{1\le s_j\le S_j} \mu^{(j)}_{s_j} }_{ 1 } \sum_{1\le s_k\le S_k} \mu^{(k)}_{s_k} \frac{\nabla p^{(k)}_{s_k}(\theta)}{p^{(k)}_{s_k}(\theta)}
= \nabla^{\mu^{(1)}} F_1(\theta) + \dots + \nabla^{\mu^{(K)}} F_K(\theta) \ ,
\end{multline*}
where we used the fact that $\frac{\nabla \tilde p_{s_1,\dots,s_K}}{\tilde p_{s_1,\dots,s_K}} = \frac{\nabla p^{(1)}_{s_1}}{p^{(1)}_{s_1}} + \dots + \frac{\nabla p^{(K)}_{s_K}}{p^{(K)}_{s_K}}$ \ .
    
\end{proof}

\paragraph{Cross Gradient Descent (XGD)}
The new notion of cross gradient that we have introduced in Definition \ref{def:cross_grad} motivates our proposed XGD algorithm that we now describe formally.
Given a function $F$ satisfying Assumption \ref{ass:sumLogConcave}, a hyperparameter probability distribution $\mu \in \Delta_S$ (from which all cross gradients will be ``seen''), and an initial point $\theta_0$, the XGD iterates are defined as follows: $\forall t\ge 1$,
\begin{equation}
  \label{eq:XGD}
  \boxed{
  \theta_{t} \leftarrow \theta_{t-1} - \gamma_t \nabla^{\mu} F(\theta_{t-1}) = \theta_{t-1} + \gamma_t \sum_{s=1}^S \mu_s \frac{\nabla p_s( \theta_{t-1} ) }{p_s( \theta_{t-1} )}
  } \ 
  \tag{XGD update}
\end{equation}

\subsection{Cross convexity}
\label{subsec:cross_convexity}

We are now ready to state one of our main contribution, namely a generalization of the notion of convexity to the class of functions given by the log-loss of any sum-log-concave function.

\begin{shadedlemma}[Cross convexity]
  \label{lem:cross_cvx}
  Consider a function $F : \mathbb{R}^m \mapsto \mathbb{R}$ satisfying Assumption \ref{ass:sumLogConcave}.
  \begin{enumerate}[(i)]
      \item For all $\theta,\eta \in \mathbb{R}^m$,
  \begin{equation*}
    F(\eta) - F(\theta) \ge 
    \nabla^{\eta} F(\theta)^\top ( \eta - \theta ) 
    +  D_{\text{KL}}\left( \left[ \frac{p_{s}(\eta)}{\sum_{s'} p_{s'}(\eta)} \right]_s \bigg\| \left[ \frac{p_s(\theta)}{\sum_{s'} p_{s'}(\theta)} \right]_s \right)
    \ .
  \end{equation*}
  \item More generally, for any distribution $\mu \in \Delta_S$: $\forall \theta,\eta \in \mathbb{R}^m$,
  \begin{equation*}
    F(\eta) - F(\theta) \ge 
    \nabla^{\mu} F(\theta)^\top ( \eta - \theta ) 
    +  D_{\text{KL}}\left( \mu \bigg\| \left[ \frac{p_s(\theta)}{\sum_{s'} p_{s'}(\theta)} \right]_s \right)
    - D_{\text{KL}}\left( \mu \bigg\| \left[ \frac{p_{s}(\eta)}{\sum_{s'} p_{s'}(\eta)} \right]_s \right)
    \ .
  \end{equation*}
  \end{enumerate}
\end{shadedlemma}

\begin{proof}
Let us lower bound the following quantity:
    \begin{multline}
    F(\eta) - F(\theta) 
    -  D_{\text{KL}}\left( \mu \bigg\| \left[ \frac{p_s(\theta)}{\sum_{s'} p_{s'}(\theta)} \right]_s \right)
    + D_{\text{KL}}\left( \mu \bigg\| \left[ \frac{p_{s}(\eta)}{\sum_{s'} p_{s'}(\eta)} \right]_s \right) \\
    = - \log\left( \frac{ \sum_{s=1}^S p_s(\eta) }{ \sum_{s=1}^S p_s(\theta) } \right)
    - \sum_{s=1}^S \mu_s \log\left( \frac{ p_s(\eta) \bigg/\sum_{s'} p_{s'}(\eta) }{ p_s(\theta) \bigg/\sum_{s'} p_{s'}(\theta) } \right)
    = - \sum_{s=1}^S \mu_s \log\left( \frac{ p_s(\eta) }{ p_s(\theta) } \right) \ .
  \end{multline}
  Then, by applying $S$ convexity inequalities,
  \begin{equation}
    - \sum_{s=1}^S \mu_s \log\left( \frac{ p_s(\eta) }{ p_s(\theta) } \right)
    \ge - \sum_{s=1}^S \mu_s \frac{ \nabla p_s(\theta)^\top ( \eta - \theta ) }{ p_s(\theta) }
    = \nabla^{\mu} F(\theta)^\top ( \eta - \theta ) \ .
  \end{equation}
    
\end{proof}

In particular in the case $S=1$, Lemma \ref{lem:cross_cvx} reduces to a classic convexity inequality since the KL terms are equal to zero and $\nabla^{\mu} F(\theta) = \nabla F(\theta)$.
In the general case $S\ge 2$, we say that the function $F$ is \emph{cross convex} since it satisfies a variant of convexity involving the cross gradient (instead of the gradient in standard convexity) and a ``penalization'' equal to the difference of two KL divergences.
In fact, this ``penalty'' term may be either positive or negative.
In particular for the first inequality (i), the positive sign in front of the KL strengthens the linear lower bound instead of relaxing it: this will allow us to prove the convergence of our new algorithm leveraging the cross convexity of $F$.

\subsection{Convergence analysis}

Given $\mu = (\mu_1,\dots,\mu_S) \in \Delta_S$ , let us define
\begin{equation*}
\mathcal{E}(\mu) = 
\left\{ \eta \quad \text{s.t.} \quad  \nabla^{\eta} F = \nabla^{\mu} F \right\} \ .
\end{equation*}
This subset of $\mathbb{R}^m$ contains all points $\eta$ that share the same cross gradient vector field as the one produced by $\mu$, i.e. for all $\eta \in \mathcal{E}(\mu)$, $\nabla^{\eta} F = \nabla^\mu F$. In particular, if the probability law generated by $\eta$ is equal to $\mu$, then $\eta$ belongs to $\mathcal{E}(\mu)$:
\begin{equation*}
    \left\{ \eta \quad \text{s.t.} \quad  \forall s, \ \frac{p_{s}(\eta)}{\sum_{s'} p_{s'}(\eta)} = \mu_s \right\} \subseteq \mathcal{E}(\mu) \ .
\end{equation*}

\begin{theorem}[Convergence of XGD]
  \label{thm:GD}
  Let $\mu \in \Delta_S$ ($S\ge 1$).
  Consider a function $F : \mathbb{R}^m \mapsto \mathbb{R}$ satisfying Assumption \ref{ass:sumLogConcave}, such that every partial loss function $\ell_s = -\log p_s$ is $B$-Lipschitz continuous.
  Then, for $T\ge 1$ iterations of XGD given by Eq. (\ref{eq:XGD}) with learning rates $\gamma_t > 0$, it holds: 
  \begin{equation*}
  \forall \eta \in \mathcal{E}(\mu) \ , \qquad 
    \frac{1}{ \sum_{t=1}^T \gamma_t } \sum_{t=1}^T \gamma_t \left( F(\theta_{t-1}) - F(\eta) \right) 
    \le \frac{ \| \theta_{0} - \eta \|^2 }{2 \sum_{t=1}^T \gamma_t } 
    + B^2 \frac{ \sum_{t=1}^T \gamma_t^2 }{ 2 \sum_{t=1}^T \gamma_t } \ .
  \end{equation*}
\end{theorem}

\begin{proof}

Let us consider $\eta \in \mathcal{E}(\mu)$.
For $t\ge 1$, we have:
\begin{equation*}
    \| \theta_t - \eta \|^2
    = \| \theta_{t-1} - \gamma_t \nabla^{\mu} F(\theta_{t-1}) - \eta \|^2 \\
    = \| \theta_{t-1} - \eta \|^2 - 2 \gamma_t \langle \nabla^{\mu} F(\theta_{t-1}) , \theta_{t-1} - \eta \rangle  + \gamma_t^2 \| \nabla^{\mu} F(\theta_{t-1}) \|^2 \ .
\end{equation*}
Then, the cross convexity inequality from Lemma \ref{lem:cross_cvx}-(i)
tells us that:
\begin{equation*}
    F(\eta) - F(\theta_{t-1}) \ge \nabla^{\mu} F(\theta_{t-1})^\top ( \eta - \theta_{t-1} ) + D_{\text{KL}}\left( \left[ \frac{p_s(\eta)}{\sum_{s'} p_{s'}(\eta)} \right]_s \bigg\| \left[ \frac{p_s(\theta_{t-1})}{\sum_{s'} p_{s'}(\theta_{t-1})} \right]_s \right) \ ,
\end{equation*}
which implies that
\begin{equation*}
    F(\theta_{t-1}) - F(\eta) \le \nabla^{\mu} F(\theta_{t-1})^\top ( \theta_{t-1} - \eta )  .
\end{equation*}
By using the boundedness of the gradient of each $\ell_s$, we deduce that
\begin{equation*}
     \gamma_t \left( F(\theta_{t-1}) - F(\eta) \right) 
    \le \frac{1}{2} \left( \| \theta_{t-1} - \eta \|^2  - \| \theta_t - \eta \|^2 \right)
    + \frac{1}{2} \gamma_t^2 B^2 \ .
\end{equation*}
Then, by summing these inequalities, we obtain:
\begin{equation*}
     \frac{1}{ \sum_{t=1}^T \gamma_t } \sum_{t=1}^T \gamma_t \left( F(\theta_{t-1}) - F(\eta) \right) 
    \le \frac{ \| \theta_{0} - \eta \|^2 }{2 \sum_{t=1}^T \gamma_t } 
    + B^2 \frac{ \sum_{t=1}^T \gamma_t^2 }{ 2 \sum_{t=1}^T \gamma_t } \ .
\end{equation*}

\end{proof}

\begin{remark}
  \label{ex:S1}
  In the case $S=1$, the function $F$ is convex and XGD coincides with gradient descent (GD). Theorem \ref{thm:GD}
  then corresponds to the classic analysis of GD for a convex function, and the upper bound is of order $\mathcal{O}(\log T/\sqrt{T})$ for decaying learning rates $\gamma_t$ of order $1/\sqrt{t}$ (see e.g. \cite{bach2021learning}).
  In our more general setting, the same guarantee can be obtained over any bounded subset $\mathcal{E}(\mu) \cap \mathcal{B}(\theta_0, D)$ by choosing $\gamma_t = \frac{D}{B \sqrt{t}}$.
\end{remark}

\section{Application: the checkered regression classifier}
\label{sec:CR}

This section formally introduces the checkered regression method along with a few elementary properties.
As will be seen, this sum-log-concave model allows to generalize multi-class logistic regression to non-linearly separable problems.
We start by introducing the binary version before moving to the multiclass setting.

\subsection{Definition and immediate properties}
\label{subsec:def}

Let $m\ge 1$. We start with the definition below.

\begin{definition}[Checkoid function]
  \label{def:checkoid}
  The checkoid function $\Xi_m$ is defined for all $z=(z_1,\dots,z_m)\in\mathbb{R}^m$ by\footnote{We recall that the hyperbolic tangent function is defined on $\mathbb{R}$ as $\tanh(x) = (e^x - e^{-x})/(e^x + e^{-x}) \in (-1,1)$.}
\begin{equation*}
    \Xi_m(z) = \frac{1}{2} \left( 1+\prod_{k=1}^m \tanh\left(\frac{z_k}{2}\right) \right) \ . 
\end{equation*}
\end{definition}

Obviously, the checkoid function is permutation-invariant: for any permutation $\tau$ of $\{1,\dots,m\}$,
$ \Xi_m(z_{\tau(1)},\dots,z_{\tau(m)}) = \Xi_m(z_1,\dots,z_m) $.
We now describe an important ``anti-symmetry'' property of our model that characterizes its checkered structure.

\begin{property}[Anti-symmetry]
  \label{property:antisym}
  Given $z=(z_1,\dots,z_m)\in\mathbb{R}^m$, let $k\in\{1,\dots,m\}$ and $z'=(z'_1,\dots,z'_m)$ with $z'_k=-z_k$ and $z'_j = z_j$ for $j\neq k$. Then,
    \begin{equation*}
      \Xi_m(z') = 1-\Xi_m(z) .
    \end{equation*}
\end{property}

\begin{proof}
  (i) is trivial while (ii) follows from the oddness of the hyperbolic tangent.
\end{proof}

Property \ref{property:antisym} generalizes the standard identity $\sigma(-x)=1-\sigma(x)$ with $\sigma(x)=1/(1+e^{-x})$ the sigmoid function.
Indeed in the case $m=1$, the checkoid coincides with the sigmoid.

\begin{example}[Case $m=1$: Sigmoid]
  For $m=1$, $\Xi_1(z)= (1/2)(1+\tanh(z/2)) =\sigma(z)$.
  In this case, the checkered regression (defined from the checkoid equipped with the log-loss) boils down to a logistic regression.
\end{example}

\begin{example}[Case $m=2$: Smooth XOR]
  \label{ex:xor}
  If $m=2$, one can easily show that the checkoid function is equal to
  \begin{equation*}
    \Xi_2(z_1,z_2) = \frac{1}{2} \left( 1+ \tanh\left(\frac{z_1}{2}\right) \tanh\left(\frac{z_2}{2}\right) \right) = \frac{\sigma(z_1)\sigma(z_2)}{\sigma(z_1+z_2)} \ .
  \end{equation*}
  Furthermore, $\Xi_2$ can be interpreted as a smooth XOR gate\footnote{For Boolean variables $A,B\in\{0,1\}$, $\text{XOR}(A,B)=A(1-B)+(1-A)B$.}, where the Boolean variables are replaced by continuous sigmoids. Indeed,
  \begin{equation*}
    \Xi_2(z_1,z_2) = \sigma(z_1) \sigma(z_2) + (1- \sigma(z_1)) (1- \sigma(z_2)) = \text{XOR}(\sigma(z_1), 1- \sigma(z_2)) \ ,
  \end{equation*}
  where we extended the XOR formula to continuous inputs in the interval $(0,1)$.
\end{example}

Figure \ref{fig:Xi2} illustrates the checkered pattern produced by the bivariate checkoid function ($m=2$).
Indeed, the decision regions are given by the checkerboard formed by the two coordinate axes:
\begin{equation}
\text{Sign}\left( \Xi_2( z_1, z_2 ) - \frac12 \right) = \text{Sign}( z_1 \cdot z_2 ) \ .
\end{equation}
Moreover, $\Xi_2(z_1,z_2)$ is equal to $1/2$ if and only if $z_1=0$ or $z_2=0$.

\begin{figure}[htbp]
    \centering
    \begin{minipage}{.5\textwidth}
    \centering
    \includegraphics[width=\textwidth]{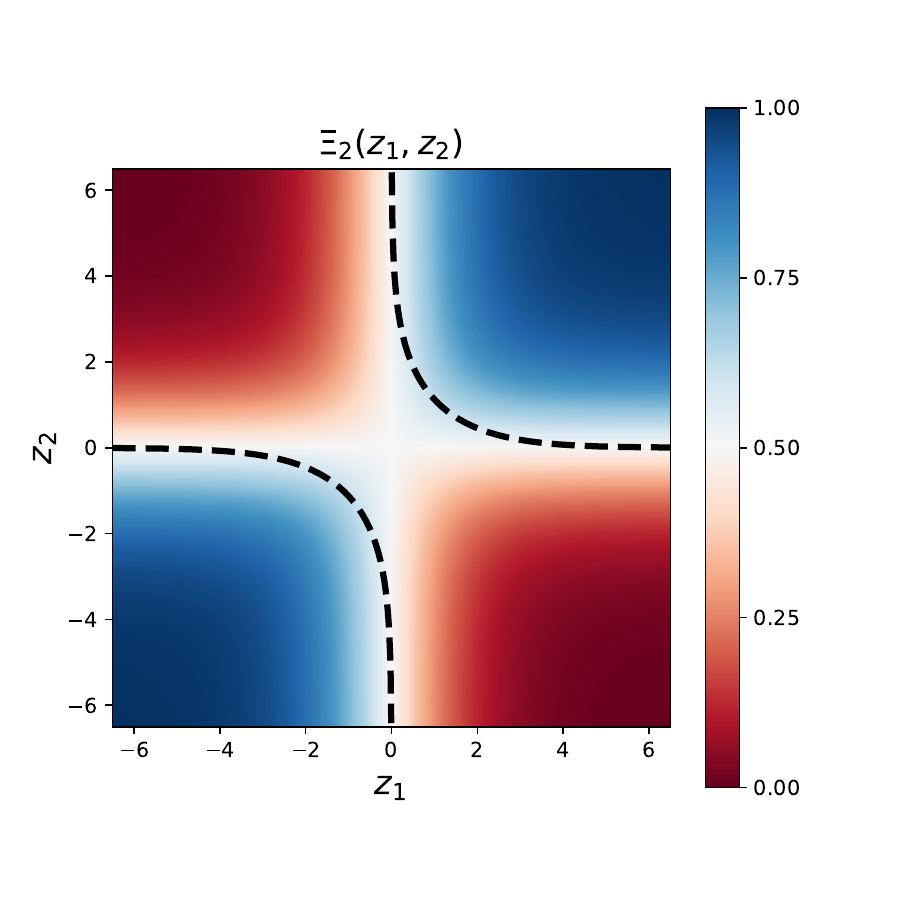}
    \end{minipage}%
    \begin{minipage}{.5\textwidth}
    \centering
    \caption{Heatmap of the bivariate checkoid function $\Xi_2(z_1,z_2)$ taking values in the interval $(0,1)$ and equal to $\frac12$ over the union of the coordinate axes. The dotted lines correspond to the frontier ``$\det(H(z_1,z_2))=0$'' (with $H$ the Hessian matrix of $-\log \Xi_2$) that forms a quartic algebraic curve in $X=e^{z_1}$, $Y=e^{z_2}$, defined by the roots of the bivariate polynomial $P(X,Y) = X^2 Y^2 - X^2 Y - X Y^2 - 2XY - X - Y + 1 = X^2 Y^2 P(1/X,1/Y)$.
    In fact, this loss function is convex in the top-right (resp. bottom-left) corner delimited by the top-right (resp. bottom-left) dotted line.
  }
    \label{fig:Xi2}
    \end{minipage}
\end{figure}

\paragraph{The checkered regression}

Equipped with the checkoid function introduced earlier,
we can now define a new non-linear model for binary classification.

\begin{definition}[Checkered regression]
  \label{def:checkreg}
  Let $\mathcal{X} \subseteq \mathbb{R}^d$ ($d\ge 1$).
  The checkered regression model with parameters $\omega=(\omega_1,\dots,\omega_m) \in \mathbb{R}^{dm}$
  is given by the posterior probabilities
\begin{equation*}
    p_\omega\left( 0 | x \right) = 1 - p_\omega\left( 1 | x \right) = \Xi_m(\omega_1^\intercal x,\dots,\omega_m^\intercal x) \quad \text{ for all } x\in\mathcal{X} \ .
\end{equation*}
\end{definition}

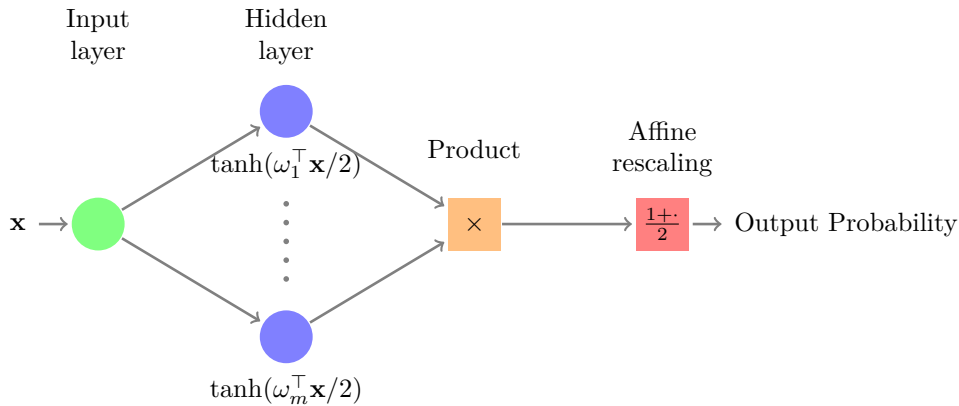
\begin{figure}[h]
\centering
\begin{tikzpicture}[shorten >=1pt,->,draw=black!50, line width=1pt, node distance=2.5cm]
    \tikzstyle{every pin edge}=[<-,shorten <=1pt]
    \tikzstyle{neuron}=[circle,fill=black!25,minimum size=20pt,inner sep=0pt]
    \tikzstyle{input neuron}=[neuron, fill=green!50];
    \tikzstyle{hidden neuron}=[neuron, fill=blue!50];
    \tikzstyle{product neuron}=[rectangle, fill=orange!50, minimum size=20pt,inner sep=0pt];
    \tikzstyle{output neuron}=[rectangle, fill=red!50, minimum size=20pt,inner sep=0pt];
    \tikzstyle{annot} = [text width=4em, text centered]

    \node[input neuron, pin=left:$\mathbf{x}$] (I-1) at (0,0) {};

    \node[hidden neuron, label=below:{$\tanh(\omega_1^\top \mathbf{x}/2)$}] (H-1) at (2.5cm,1.5 cm) {};
    \node[hidden neuron, label=below:{$\tanh(\omega_m^\top \mathbf{x}/2)$}] (H-3) at (2.5cm,-1.5 cm) {};


    \foreach \y in {-0.75,-0.5,-0.25,0,0.25}
      \node[scale=2, text=black!50] at (2.5cm,\y cm) {$\cdot$};

    \node[product neuron] (P-1) at (5cm,0cm) {$\times$};

    \node[output neuron,pin={[pin edge={->}]right:Output Probability}, right of=P-1] (O-1) {$\frac{1+\cdot}{2}$};

    \foreach \dest in {1,3}
        \path (I-1) edge (H-\dest);

    \foreach \source in {1,3}
        \path (H-\source) edge (P-1);

    \path (P-1) edge (O-1);

    \node[annot,above of=H-1, node distance=1cm] (hl) {Hidden layer};
    \node[annot,left of=hl] {Input layer};
    \node[annot,above of=P-1, node distance=1cm] {Product};
    \node[annot,above of=O-1, node distance=1cm] {Affine rescaling};
\end{tikzpicture}
\caption{Diagram of the checkered regression neural network architecture for binary classification.
It can be seen as a 1-hidden layer neural network with $tanh$ activations that are multiplied together.
Lastly, the product is rescaled to the interval $(0,1)$. In particular, the case $m=1$ simply outputs a sigmoid.}
\label{fig:diagram}
\end{figure}

In particular, the ``anti-symmetry'' Property \ref{property:antisym} implies that the decision boundary of our model showcases a checkered structure as explained in the next proposition.

\begin{proposition}[Hamming distance parity]
  \label{prop:hamming}
  Consider a checkered regression (CR) model with parameters $\omega=(\omega_1,\dots,\omega_m)$.
  Let $x$ and $x'$ be two points in $\mathbb{R}^d$ both outside the $m$ hyperplanes of the CR model, i.e. such that $\omega_k^\intercal x\neq 0$ and $\omega_k^\intercal x'\neq 0$ for all $1\le k\le m$.
  Then, the two following conditions are equivalent.
  \begin{enumerate}[(i)]
  \item The CR model predicts that $x$ and $x'$ share the same label:
  \begin{equation*}
  \text{sign}\left(p_\omega\left( 0 | x \right) - \frac{1}{2} \right)
  = \text{sign}\left(p_\omega\left( 0 | x' \right) - \frac{1}{2} \right) \ .
\end{equation*}
  \item The Hamming distance $\sum_{k=1}^m \mathbb{I}\left\{ \text{sign}\left( \omega_k^\intercal x \right) \neq \text{sign}\left( \omega_k^\intercal x' \right) \right\}$ is even.
\end{enumerate}
\end{proposition}

\begin{proof}
  By successive applications of Property \ref{property:antisym}.
\end{proof}

Proposition \ref{prop:hamming} shows that the decision regions of a CR model form a hyperplane tessellation of the feature space $\mathcal{X}$, where the tiles are binary labeled in a checkered fashion.

\begin{remark}[DC loss]
We point out that the log-loss of our model is not convex. In fact, it is equal to the difference of two convex functions (a.k.a. ``DC'').
In particular for $m=2$,
\begin{equation}
\label{eq:DC}
-\log( \Xi_2(z_1,z_2) ) = \underbrace{[ -\log(\sigma(z_1)\sigma(z_2)) ]}_{\text{convex}} - \underbrace{[- \log(\sigma(z_1+z_2)) ]}_{\text{convex}} .
\end{equation}
For general $m\ge 2$ (and also in the multiclass case defined later), the log-loss of the checkered regression model is equal to the difference of two convex ``LogSumExp'' functions, and thus is also DC.
\end{remark}

\begin{remark}[XGD vs constrained GD]
We highlight that XGD is intimately related to a constrained GD procedure.
For simplicity, let us focus on the minimization of the function $\ell(z_1,z_2) = -\log( \Xi_2(z_1,z_2) ) = -\log(\sigma(z_1)\sigma(z_2)) + \log(\sigma(z_1+z_2))$.
For any real constant $\kappa$, the function $\ell$ is convex over the region ``$z_1+z_2 =\kappa$''. Hence, it can be minimized by gradient descent with respect to a single variable, say $z_1$, through the substitution $z_2=\kappa-z_1$.
By selecting $\mu=(\mu_1,1-\mu_1)$ with $\mu_1=\sigma(\kappa)$, and \emph{starting from any initial point} (even outside the ``$\kappa$-region''), XGD allows to minimize $\ell$ over that same region. Although the benefit compared to constrained GD seems limited in this simple example (since the substitution $z_2=\kappa-z_1$ is trivial), XGD becomes much more convenient in more complex scenarios involving large and high dimensional datasets. Indeed, XGD circumvents the double need of finding an initial point belonging to $\mathcal{E}(\mu)$, and of performing intricate projections onto this same set, which can be very challenging in general (especially for checkered regression models with $m>2$).
\end{remark}

\subsection{Sum-log-concavity via circular convolution}
\label{subsec:Sum-log-concavity}

The purpose of this subsection is to demonstrate that the checkered regression is a sum-log-concave model.

\begin{proposition}[Circular convolution]
  \label{prop:circonv}
  For all $z=(z_1,\dots,z_m)\in\mathbb{R}^m$,
  \begin{equation*}
    \begin{pmatrix} \Xi_m(z) \\ 1-\Xi_m(z) \end{pmatrix}
      = \begin{pmatrix} \sigma(z_1)\\ 1-\sigma(z_1) \end{pmatrix} \circledast \dots \circledast \begin{pmatrix} \sigma(z_m)\\ 1-\sigma(z_m) \end{pmatrix} \ ,
  \end{equation*}
  where $\circledast$ denotes the circular convolution operator.
\end{proposition}

\begin{proof}
  Let $m\ge 1$ and $z=(z_1,\dots,z_m)\in\mathbb{R}^m$.
  For all $1 \le k \le m$, let us denote the following binary probability laws:
  \begin{equation}
    \label{eq:bk}
      b_k = \begin{pmatrix} \sigma(z_k)\\ 1-\sigma(z_k) \end{pmatrix} \ ,
  \end{equation}
  with $\sigma$ the sigmoid function.
  Then, the \emph{Discrete Fourier Transform} (DFT in short) 
  $B_k$ of each $b_k$ is given by:
  \begin{equation}
    \label{eq:Bk}
      B_k = \text{DFT}(b_k) = \begin{pmatrix} \sigma(z_k) + 1-\sigma(z_k) \\ \sigma(z_k) - (1-\sigma(z_k)) \end{pmatrix}
      = \begin{pmatrix} 1 \\ 2\sigma(z_k) - 1 \end{pmatrix}
      = \begin{pmatrix} 1 \\ \tanh(\frac{z_k}{2}) \end{pmatrix} \ .
  \end{equation}

  By the circular convolution theorem, we know that the circular convolution can be computed through the product of the discrete Fourier transforms.
  Thus, we consider the componentwise product of the $B_k$'s:
  \begin{equation}
    \label{eq:prodBk}
      \prod_{k=1}^m B_k = \begin{pmatrix} 1 \\ \prod_{k=1}^m \tanh(\frac{z_k}{2}) \end{pmatrix} \ .
  \end{equation}

  And finally, we take the \emph{Inverse Discrete Fourier Transform} (IDFT):
  \begin{equation}
    \label{eq:IDFT}
      \text{IDFT}(\prod_{k=1}^m B_k ) =
      \frac{1}{2}\begin{pmatrix} 1+\prod_{k=1}^m \tanh(\frac{z_k}{2}) \\ 1-\prod_{k=1}^m \tanh(\frac{z_k}{2}) \end{pmatrix}
      = \begin{pmatrix} \Xi_m(z) \\ 1-\Xi_m(z) \end{pmatrix} \ ,
  \end{equation}
  which is the binary law produced by the checkoid function $\Xi_m$.

\end{proof}

In the next section, we generalize our checkered regression model to the multiclass setting.
More precisely, we extend Proposition \ref{prop:circonv} by replacing sigmoids with SoftArgMax vectors.


\paragraph{Probabilistic interpretation}

The next result describes a probabilistic interpretation of the checkoid function in terms of the parity of the number of successes in Bernoulli trials.
Although this is a corollary of Proposition \ref{prop:circonv}, we provide an additional more intuitive proof.

Let $X_1,\dots,X_m$ be independent Bernoulli random variables valued in $\{-1,1\}$ with success probabilities:
\begin{equation}
  \label{eq:coins}
  \mathbb{P}(X_k=1) = \sigma(z_k).
\end{equation}
For any $1 \le k \le m$, we refer to the event $(X_k=-1)$ as a ``failure''.

\begin{proposition}
  \label{prop:even}
Let us flip the $m$ independent Bernoulli coins defined in Eq. \ref{eq:coins}. Then,
\begin{equation*}
  \mathbb{P}(\text{total number of failures is even})
  = \Xi_m(z_1,\dots,z_m) \ . 
\end{equation*}
\end{proposition}

\begin{proof}

\begin{multline}
\mathbb{P}(\text{total number of failures is even})
= \mathbb{P}(X_1\times\dots\times X_m = 1) \\
= \frac{1}{2} \left( 1 + \mathbb{E}[X_1\times\dots\times X_m] \right) \quad (\text{since } X_1\times\dots\times X_m \text{ is valued in } \{-1,1\}) \\
= \frac{1}{2} \left( 1 + \mathbb{E}[X_1]\times\dots\times \mathbb{E}[X_m] \right) \quad (\text{by independence}) \\
= \frac{1}{2} \left( 1 + (2\sigma(z_1)-1)\times\dots\times (2\sigma(z_m)-1) \right)
= \frac{1}{2} \left( 1 + \prod_{k=1}^m \tanh(\frac{z_k}{2}) \right) \, .
\end{multline}

\end{proof}

Proposition \ref{prop:even} shows that the quantity $\Xi_m(z)$ is a very natural one that results from combining several sigmoids.

\paragraph{GMM class posterior probability}

Furthermore, the checkoid function can be derived as the class posterior probability of a Gaussian mixture model (GMM).

\begin{proposition}[Checkered GMM posterior]
\label{prop:posterior}
    Let $d,m\ge 1$, $\mu_0,\mu_1,\dots,\mu_m \in \mathbb{R}^d$ (such that $\mu_1,\dots,\mu_m$ are pairwise orthogonal) and denote the matrix $\mu=(\mu_1,\dots,\mu_m) \in \mathbb{R}^{d\times m}$. Consider a random pair $(X,Y)$ valued in $\mathbb{R}^d \times \{0,1\}$ such that $\mathbb{P}(Y=1)=\frac12$ and, given $Y$,
    \begin{equation*}
        X \sim \frac{1}{2^{m-1}} \sum_{\upsilon\in\{0,1\}^m : |\upsilon| \equiv Y [2] }  \mathcal{N}( \mu_0 + \mu \upsilon, I_d ) \ .
    \end{equation*}
    Then for all $x\in \mathbb{R}^d$,
    \begin{equation*}
        \mathbb{P}(Y=+1 | X=x) = \Xi_m( w_1^\top x + b_1, \dots, w_m^\top x + b_m) \ ,
    \end{equation*}
    where $w_1,\dots,w_m, b_1,\dots,b_m$ are constants depending on the Gaussian parameters $\mu_0, \mu_1,\dots,\mu_m$.
\end{proposition}

\begin{proof}
    The proof easily follows from:
    \begin{equation*}
        \Xi_m(z) = \frac{\sum_{\upsilon \in \{0,1\}^m : |\upsilon| \equiv 0 [2] } e^{-\upsilon^\top z} }{(1+e^{-z_1})\times \dots \times (1+e^{-z_m})}
        \quad \text{and} \quad
        1-\Xi_m(z) = \frac{\sum_{\upsilon \in \{0,1\}^m : |\upsilon| \equiv 1 [2] } e^{-\upsilon^\top z} }{(1+e^{-z_1})\times \dots \times (1+e^{-z_m})} \ .
    \end{equation*} 
\end{proof}

Of course, a more general version of Proposition \ref{prop:posterior} can be proved in the same way for mixtures of distributions belonging to any exponential family. In the multiclass case, one shall replace the set $\{0,1\}^m$ by $\{0,\dots,c-1\}^m$ and the ``modulo 2'' by ``modulo $c$''.
Note that the case $m=1$ corresponds to the well-understood linearly separable case with a single Gaussian per class.
The more complex case $m=2$ with two Gaussian distributions per class is often called ``XOR GMM'' (for instance in \cite{ben2022high}): see Figure \ref{fig:cloud} for an illustration.

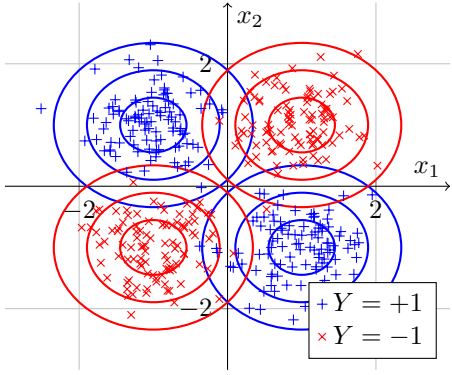
\begin{figure}[htbp]
    \centering
    \begin{minipage}{.5\textwidth}
    \centering
    \begin{tikzpicture}
        \begin{axis}[
            xmin=-3, xmax=3,
            ymin=-3, ymax=3,
            grid=both,
            grid style={line width=.1pt, draw=gray!30},
            major grid style={line width=.2pt,draw=gray!50},
            axis lines=center,
            axis line style={->},
            xlabel={$x_1$},
            ylabel={$x_2$},
            legend pos=south east,
            ]

            \addplot[only marks, mark=+, color=blue] table {figures/points1_standard.dat};
            \addlegendentry{$Y=+1$}
            \draw[blue, thick, rotate around={0:(axis cs:1,-1)}] (axis cs:1,-1) ellipse [x radius=0.4472135954999579, y radius=0.4472135954999579];
            \draw[blue, thick, rotate around={0:(axis cs:-1,1)}] (axis cs:-1,1) ellipse [x radius=0.4472135954999579, y radius=0.4472135954999579];
            \draw[blue, thick, rotate around={0:(axis cs:1,-1)}] (axis cs:1,-1) ellipse [x radius=0.8944271909999159, y radius=0.8944271909999159];
            \draw[blue, thick, rotate around={0:(axis cs:-1,1)}] (axis cs:-1,1) ellipse [x radius=0.8944271909999159, y radius=0.8944271909999159];
            \draw[blue, thick, rotate around={0:(axis cs:1,-1)}] (axis cs:1,-1) ellipse [x radius=1.3416407864998738, y radius=1.3416407864998738];
            \draw[blue, thick, rotate around={0:(axis cs:-1,1)}] (axis cs:-1,1) ellipse [x radius=1.3416407864998738, y radius=1.3416407864998738];

            \addplot[only marks, mark=x, color=red] table {figures/points2_standard.dat};
            \addlegendentry{$Y=-1$}
            \draw[red, thick, rotate around={0:(axis cs:1,1)}] (axis cs:1,1) ellipse [x radius=0.4472135954999579, y radius=0.4472135954999579];
            \draw[red, thick, rotate around={0:(axis cs:-1,-1)}] (axis cs:-1,-1) ellipse [x radius=0.4472135954999579, y radius=0.4472135954999579];
            \draw[red, thick, rotate around={0:(axis cs:1,1)}] (axis cs:1,1) ellipse [x radius=0.8944271909999159, y radius=0.8944271909999159];
            \draw[red, thick, rotate around={0:(axis cs:-1,-1)}] (axis cs:-1,-1) ellipse [x radius=0.8944271909999159, y radius=0.8944271909999159];
            \draw[red, thick, rotate around={0:(axis cs:1,1)}] (axis cs:1,1) ellipse [x radius=1.3416407864998738, y radius=1.3416407864998738];
            \draw[red, thick, rotate around={0:(axis cs:-1,-1)}] (axis cs:-1,-1) ellipse [x radius=1.3416407864998738, y radius=1.3416407864998738];
        \end{axis}
    \end{tikzpicture}
    \end{minipage}%
    \begin{minipage}{.5\textwidth}
    \centering
    \caption{Cloud of points $X=(x_1,x_2)$ from two classes, each generated from a mixture of Gaussian distributions with identical covariance matrix equal to 0.2 times the 2-dimensional identity matrix. The positive class (blue $+$) is a balanced mixture of two Gaussians
    with respective means at $(1, -1)$ and $(-1, 1)$. The negative class (red $\times$) is also a balanced mixture of two Gaussians with respective means at $(1,1)$ and $(-1,-1)$.}
    \label{fig:cloud}
    \end{minipage}
\end{figure}

\subsection{Multiclass checkered regression}
\label{subsec:multi}

We now define a muticlass version of our model that can be used in classification problems having more than two classes.
Let $\mathcal{Y}=\{0,\dots,c-1\}$ be the set of classes with $c\ge 2$.

\begin{definition}[Multiclass checkered regression]
  \label{def:MCR}
  The multiclass checkered regression model with parameters $\Omega=(\Omega_1,\dots,\Omega_m) \in \mathbb{R}^{m \times c \times d}$ (with each $\Omega_k \in \mathbb{R}^{c \times d}$)
  is given, for any input vector $x \in \mathbb{R}^d$, by the posterior probabilities
\begin{equation*}
    \forall y\in\mathcal{Y}\, , \quad p_\Omega\left( y | x \right) = \Xi_{m,y}(\Omega_1 x,\dots,\Omega_m x) \ ,
\end{equation*}
where the multiclass checkoid function $\Xi_{m,y}$ is defined for all $Z=(Z_1^\top,\dots,Z_m^\top)^\top \in \mathbb{R}^{m \times c }$ as

\begin{equation*}
  \vec \Xi_{m}(Z) =
  \begin{pmatrix} \Xi_{m,0}(Z) \\ \vdots \\ \Xi_{m,c-1}(Z) \end{pmatrix}
  = \vec \sigma(Z_1) \circledast \dots \circledast \vec \sigma(Z_m) \ ,
\end{equation*}
where $\circledast$ denotes the circular convolution operator and we recall that the SoftArgMax function is defined at equation \ref{eq:softargmax}.

\end{definition}

The multiclass checkered regression model is a simple way of obtaining a sum-log-concave probability distribution
by taking the circular convolution of several SoftArgMax vectors. As in the binary case discussed before, it allows to represent
richer structures than a standard multiclass logistic regression that only uses a single SoftArgMax distribution (particular case $m=1$).

\paragraph{(Cross) Gradient Formulas}

We here provide the expressions for the (cross) gradients of the log-loss of the checkered regression model.

\begin{proposition}[Gradient]
  \label{prop:grad}
  \begin{enumerate}[(i)]
      \item Consider the logarithmic loss $\ell(z) = -\log(\Xi_m(z))$ for all $z=(z_1,\dots,z_m)\in \mathbb{R}^m$.
  Then for any $1\le k \le m$, the $k$-th partial derivative of $\ell$ is
  \begin{equation*}
    \frac{\partial \ell}{\partial z_k}(z) = \sigma(z_k) \cdot \left( 1 - \frac{\Xi_{m-1}(z_{-k})}{\Xi_m(z)} \right) \ ,
  \end{equation*}
  where $z_{-k}=(z_j)_{j\neq k}$ .
  \item For multiclass,
  \begin{equation*}
    - \nabla_{Z_k} \log \Xi_{m,y}(Z) = \left[ \sigma_l(Z_k)\left( 1 - \frac{ \Xi_{m-1, y-l}(Z_{-k})}{\Xi_{m,y}(Z)} \right) \right]_{0\le l\le c-1} \ ,
  \end{equation*}
where $Z_{-k}=(Z_j)_{j\neq k}$
and $y-l$ must be understood ``modulo $c$''.
  \end{enumerate}
\end{proposition}

In particular, these partial derivatives are all bounded in the interval $(-1,1)$ since
\begin{equation*}
    \sum_{0\le l\le c-1} \sigma_l(Z_k)  \Xi_{m-1, y-l}(Z_{-k}) = \Xi_{m,y}(Z) .
\end{equation*}

\begin{proposition}[Cross-Gradient]
  \label{prop:crossgrad}
  For $Z=(Z_1^\top,\dots,Z_m^\top)^\top, Z'=(Z_1'^\top,\dots,Z_m'^\top)^\top \in \mathbb{R}^{m \times c }$,
    \begin{equation*}
    - \nabla^{Z'}_{Z_k} \log \Xi_{m,y}(Z) =
    \sum_{j=0}^{c-1} \frac{\sigma_{j}(Z'_k) \Xi_{m-1,y-j}(Z'_{-k})}{\Xi_{m,y}(Z')}
    \begin{bmatrix}
        \sigma_0(Z_k) \\
        \vdots \\
        \sigma_j(Z_k) - 1 \\
        \vdots \\
        \sigma_{c-1}(Z_k)
    \end{bmatrix} \ ,
  \end{equation*}
  which is equal to the gradient in Proposition \ref{prop:grad}-(ii) for $Z'=Z$.
\end{proposition}

\begin{proof}
    We have
    \begin{multline*}
        - \nabla^{Z'}_{Z_k} \log \Xi_{m,y}(Z)
        = \frac{1}{\Xi_{m,y}(Z')}
        \sum_{\upsilon \in \{0,\dots,c-1\}^m : |\upsilon|\equiv y [c]} \prod_{l=1}^m \sigma_{\upsilon_l}(Z_l') \begin{bmatrix}
        \sigma_0(Z_k) \\
        \vdots \\
        \sigma_{\upsilon_k}(Z_k) - 1 \\
        \vdots \\
        \sigma_{c-1}(Z_k)
    \end{bmatrix} \\
    = \frac{1}{\Xi_{m,y}(Z')}
    \sum_{j=0}^{c-1} \sigma_{j}(Z_k')
    \underbrace{\sum_{\tilde \upsilon \in \{0,\dots,c-1\}^{m-1} : |\tilde \upsilon|\equiv y-j [c]} \prod_{l\neq k} \sigma_{\tilde \upsilon_l}(Z_l')}_{ \Xi_{m-1, y-j}(Z_{-k}') }
    \begin{bmatrix}
        \sigma_0(Z_k) \\
        \vdots \\
        \sigma_j(Z_k) - 1 \\
        \vdots \\
        \sigma_{c-1}(Z_k)
    \end{bmatrix}.
    \end{multline*}
\end{proof}

\section{Conclusion}
\label{sec:conclusion}

In this paper, we proposed an extension of the standard convex optimization framework for the class of functions given by the log-loss of a sum-log-concave function. We also proposed a sum-log-concave generalization of logistic regression.
Future directions of research include:
\begin{itemize}
    \item Exploring adaptive strategies for choosing the hyperparameter distribution $\mu$ in XGD, potentially through Dirichlet distribution random sampling.
    \item Combining XGD with GD, leveraging XGD's ability to avoid saddle points with a suitable $\mu$, and GD's exploration of better minimizers beyond the $\mu$-domain.
    \item Extending Fenchel's duality and biconjugate theorems to accommodate the sum-log-concave scenario.
\end{itemize}

\subsubsection*{Acknowledgments}
The author thanks Nidham Gazagnadou for precious insights on convex optimization.

\bibliography{bib_tmlr}
\bibliographystyle{tmlr}

\nocite{achab2022checkered}

\end{document}